\documentclass[12p]{amsart}

\usepackage{amssymb,latexsym, amsmath, amscd, array, graphicx}
\usepackage[utf8]{inputenc}
\usepackage{tikz-cd}
\usepackage[titletoc]{appendix}

\swapnumbers
\numberwithin{equation}{section}

\def\m{\medskip}

\newtheorem{thm}{Theorem}[section]
\newtheorem{lem}[thm]{Lemma}
\newtheorem{lemma}[thm]{Lemma}
\newtheorem{prop}[thm]{Proposition}
\newtheorem{propdef}[thm]{Proposition-Definition}
\newtheorem{cor}[thm]{Corollary}

\newtheorem{rem}[thm]{Remark}
\newtheorem{remark}[thm]{Remark}
\newtheorem{defin}[thm]{Definition}
\newtheorem{df}[thm]{Definition}
\newtheorem{example}[thm]{Example}
\newtheorem{question}[thm]{Question}
\newcommand\theoref{Theorem~\ref}
\newcommand\lemref{Lemma~\ref}
\newcommand\propref{Proposition~\ref}

\newcommand\defref{Definition~\ref}

\newcommand\secref{Section~\ref}

\newcommand{\SU}{{\mathcal{U}}}
\newcommand{\ZZ}{\mathbb{Z}}
\newcommand{\CC}{\mathbb{C}}

\newcommand{\RR}{\mathbb{R}}

\newcommand{\inc}{\hookrightarrow}
\newcommand{\la}{\langle}
\newcommand{\ra}{\rangle}
\newcommand{\x}{\times}

\newcommand{\ov}{\overline}

\def\ga{\alpha} 
\def\gb{\beta}

\def\go{\omega}

\def\RR{\mathbb R}
\def\ZZ{\mathbb Z}
\def\s1{S^1}
\def\ms1{M/S^1}

\def\cat{\operatorname{cat}}
\def\wgt{\operatorname{wgt}}

\def\lcm{\operatorname{lcm}}
\def\gcd{\operatorname{gcd}}
\def\ts{\times}
\def\ds{\displaystyle}

\long\def\forget#1\forgotten{} %

\begin{document}

\title[Symplectically Aspherical Orbifolds]{Symplectic Asphericity, Category Weight, and 
Closed Characteristics of K-Contact Manifolds}

\author{Yuli Rudyak} 
\address{Department of Mathematics\\
University of Florida\\
Gainesville, Florida 32608, USA}

\author{Aleksy Tralle} 
\address{Department of Mathematics and Computer Sciences\\
University of Warmia and Mazury\\
S\l onczna 54, 10-710, Olsztyn\\
Poland}

\maketitle
\renewcommand{\thefootnote}{\fnsymbol{footnote}} 

\begin{abstract}
Let $M$ be a closed K-contact $(2n+1)$-manifold equipped with a quasi-regular K-contact structure. Rukimbira~\cite{Ruk1} proved that the Reeb vector field $\xi$ of this structure has at least $n+1$ closed characteristics. We note that $\xi$ has at least $2n+1$ closed characteristics provided that the space of  leaves of the foliation determined by $\xi$ is symplectically aspherical.
\end{abstract}

\tableofcontents

\section{Introduction} 

\begin{propdef}[\cite{BG}]\label{d:cont}\rm Given a  $(2n+1)$-dimensional manifold $M$, a {\em contact form} on $M$ is a 1-form $\ga$ on $M$ such that the inequality $\ga\land(d\ga)^n\neq 0$ is valid everywhere on $M$. In this case there is a unique vector field $\xi$, satisfying the equalities $\ga(\xi)=1$ and $i_{\xi}(d\ga)=0$. This vector field is called the {\em Reeb vector field}. An integral curve of the Reeb vector field $\xi$ is called a {\it characteristic} of the contact manifold $(M,\ga)$.
\end{propdef}

\begin{df}[see~\cite{MRT}]\rm 
A {\em K-contact structure} is a quintuple $(M,\ga,\xi,\Phi,g)$ such that
\begin{itemize}
\item $(M, \ga)$ is a contact form as in \defref{d:cont}; 
\item $\xi$ is the Reeb vector field for $\ga$; 
\item $\Phi$ is an endomorphism $\Phi$ of $TM$ (a tensor of type (1,1)) such that
$\Phi^2=-\operatorname{id}+\xi\otimes\ga$;
\item $d\ga(\Phi X,\Phi Y)=d\ga(X,Y)$ for all $X,Y$;
\item $d\ga(\Phi X,X)>0$ for all nonzero $X\in\operatorname{Ker}\ga$;
\item the Reeb vector field $\xi$ is a Killing vector field with respect to the Riemannian metric defined by the formula 
\[
g(X,Y)=d\ga(\Phi X,Y)+\ga(X)\ga(Y).
\]
\end{itemize}
\end{df}

Let $(M,\ga,\xi,\Phi,g)$ be a K-contact manifold. Consider the contact cone as the Riemannian manifold
$C(M):=(M\times\mathbb{R}^{>0},t^2g+dt^2)$.
Define the almost complex structure $I$ on $C(M)$ by:
 \begin{itemize}
\item $I(X)=\Phi(X)$ on $\operatorname{Ker}\ga$,
\item $I(\xi)=t{\partial\over\partial t},\,I(t{\partial\over\partial t})=-\xi$, for the Killing vector field $\xi$ of $\ga$.
\end{itemize}
The K-contact manifold $(M,\ga,\Phi,\xi, g)$ is {\it Sasakian} if $I$ is integrable.

\m The theory of such structures appeared to be very important in geometry and its applications, especially after the publishing of the fundamental monograph by Boyer and Galicki \cite{BG}. In this monograph, the authors presented a research program of studying topological properties of K-contact manifolds. Recently, there were published  several works on different aspects of the topology of K-contact and Sasakian manifolds \cite{BFMT}, \cite{MT}, \cite{MRT}. One of the early results on topological properties of K-contact manifolds was obtained by Rukimbira \cite{Ruk1} who proved that the Reeb vector field on any K-contact manifold has at least $n+1$ closed characteristics. In this work we discuss the following question: {\it is this estimate  best possible?} We show that under the additional assumption of symplectic asphericity of the space of leaves of the foliation determined by the Reeb field, the estimate can be improved to $2n+1$. 

\m Recall that a manifold (or orbifold) $X$ with symplectic form $\go$ is {\em symplectically aspherical} if $[\go]|_{\pi_2(X)}=0$; here $[\omega]$ is the de Rham cohomology class of the form $\omega$. The condition $[\go]|_{\pi_2(X)}=0$ appeared originally in the papers of Floer~\cite{F} and Hofer~\cite{H} in the context of Lagrangian intersection theory, where symplectic asphericity 
ensures the absence of non-trivial pseudo-holomorphic spheres in $X$.

\m Floer and Hofer used the advantages of {\em analytic} properties of 
such manifolds.  However, it turned out that symplectically aspherical manifolds have nice and controllable {\em homotopy} properties.  Later it appeared to be clear that most of these properties are based on the fact that, in the case of symplectic asphericity, the category weight of $[\omega]$ equals 2, cf. \lemref{l:wgt2}. It was an important ingredient in proving of the Arnold conjecture about symplectic fixed points~\cite{Rud2, RO} and in an improvement  of the estimate of the number of closed orbits of charged particles in symplectic magnetic fields~\cite{RT}. In our paper we exploit this approach in \secref{s:wgt}.

\m  In~\cite {RT} the authors expressed the hope that the technique of category weight potentially may have many applications to problems in symplectic topology, and, in a broader sense, to other non-linear analytic problems. In this paper we, in a sense, equal the hopes by applying the theory of category weight to a special problem in the topology of contact metric manifolds. 

\m The paper is organized as follows. In \secref{s:prel} we collect preliminary information on closed characteristics of contact manifolds. In \secref{s:seifert} we present the necessary information on K-contact structures. In \secref{s:wgt} we prove the main theorem of manifolds with $\geq 2n+1$ closed characteristics, and in \secref{s:2n+1} we give some examples of such manifolds. In \secref{s:ruk} we outline some of the possible direction of research in the area.  In appendices we remind basic information on orbifolds (following~\cite{BG, MRT} and Sullivan model theory (following~\cite{FHT}).

 \m {\bf Acknowledgment:} The first author was partially supported by a grant from the Simons Foundation (\#209424 to Yuli Rudyak). The second author thanks the Department of Mathematics of the University of Florida at Gainesville for hospitality.

\section{Preliminaries}\label{s:prel}

The section is an extract  from~\cite{Ruk1, BR}.

\begin{df}[\cite{CLOT}] \rm Given a map $\varphi: A \to Y$, we say that a subset $U$ of $A$ is {\it $\varphi$-categorical} if it is open (in $A$) and $\varphi|U$ is null-homotopic. We define the {\it Lusternik--Schnirelmann category} $\cat \varphi $ of $\varphi$ as follows:
$$
\cat \varphi=\min\{k\bigm|A=U_o\cup U_1\cup \cdots \cup U_{k+1} \text{ where each $U_i$ is $\varphi$-categorical}\}.
$$
 Furthermore, we define the {\it Lusternik--Schnirelmann category} $\cat Y$ of a space $Y$ by setting  $\cat X:= \cat 1_X$.
\end{df}

Consider a smooth action on a compact Lie group $G$ on a closed smooth manifold $M$ and let $f: M\to \RR$ be a smooth $G$-equivariant function. Clearly, if $p$ is a critical point of $f$ then every point $gp, g\in G$ is critical. So, each critical point $p$ produces a {\it critical orbit}.
 
\begin{prop}\label{p:cat} The number of critical orbits of $f$ is at least $1+\cat (M/G)$.
\end{prop}
  
\begin{proof} Consider the gradient flow of $f$ with respect to a $G$-invariant Riemannian metric on $M$. The flow induces a (gradient-like) flow $\Phi$ on $M/G$, and critical obits of $f$ are exactly the rest points of $\Phi$. Now the result follows from~\cite[Theorem 1.82]{CLOT}. 
One can  also consult~\cite{W}.
\end{proof}

Let $M$ be a $(2n+1)$-dimensional compact manifold equipped with an effective $S^1$-action $a: S^1 \ts M \to M$. Let $\ga$ be a circle invariant contact form on $M$. Let $Z$ be the infinitesimal generator of the $S^1$-action. Consider the function $S=i_Z(\ga)$ on $M$ and suppose that $S\neq 0$ everywhere on $M$.

\begin{prop}\label{p:closed} 
{\rm (i)} We have $dS=-i_Zd\ga$;\\
{\rm (ii)} if $p$ is a critical point of $S$ then the $Z$-orbit through $p$ is a characteristic;\\
{\rm (iii)} the  Reeb vector field $\xi$ has at least $1+\cat (\ms1)$ closed characteristics. 
\end{prop}    

\begin{proof} (i) Since $\ga$ is circle invariant, $L_Z\ga=0$. So, $i_Zd\ga+di_Z\ga=0$, and hence
$
dS=di_Z\ga=-i_Zd\ga.
$

\smallskip
(ii) If $p$ is a critical point of $S$ then $dS_p=0$, and so $i_Zd\ga=0$ at $p$. Since $i_{\xi}d\ga\equiv 0$, we conclude that $Z_p$ and $\xi_p$ are proportional (because $d\ga$ has one-dimensional kernel). Hence the $Z$-orbit through $p$ is a characteristic. 

\smallskip
(iii) Since all $Z$-orbits are closed, every critical orbit is a closed characteristic, and the result follows from \propref{p:cat}.
\end{proof}

\begin{lem}\label{l:form} 
 Consider the $1$-form $\gb=\ds\frac{\ga}{S}$. The form $\gb$ is contact, and 
\[
i_Z(d\gb)=0.
\]
Furthermore, the orbifold $M/S^1$ possesses a symplectic differential form $\omega$.
\end{lem}

\begin{proof} We have 
\[
d\gb=\frac{-1}{S^2}dS\wedge \ga+\frac{1}{S}d\ga
\]
Hence
\[
\gb\wedge (d\gb^n)=\ga\wedge (d\ga)^nS^{-n}\neq 0
\]
 everywhere on $M$, and hence $\gb$ is contact.
Now (we use \propref{p:closed}(i))
 \[
d\gb=\frac{-1}{S^2}dS\wedge \ga+\frac{1}{S}d\ga=\frac{-1}{S^2}d(i_Z\ga)\wedge \ga+\frac{1}{S}d\ga=\frac{1} {S^2}(i_Zd\ga)\wedge \ga+\frac{1}{S}d\ga
\] 
Hence, 
\[
i_Zd\gb=\frac{1}{S^2}(i_Zi_Zd\ga+(-1)i_Zd\ga\wedge i_Z\ga)+\frac{1}{S}i_Zd\ga
=0+\frac{-1}{S^2}Si_Zd\ga+\frac{1}{S}i_Zd\ga=0.
\]
 Since $i_Zd\gb=0$, the differential form $d\gb$ yields the differential form $\omega$ on the orbifold $\ms1$, see~\cite[Proposition B]{W}. Moreover,  $\omega$ is a symplectic form, see loc. cit.  
\end{proof}

\begin{thm}
 The vector field $\xi$ has at least $n+1$ closed characteristics.
\end{thm}

\begin{proof} Let $\omega$ be the above-mentioned symplectic form. Then, because of the de Rham Theorem for orbifolds, we have a cohomology class $[\omega]\in H^2(\ms1;\RR)$, and $[\omega]^n\ne 0$. So, the cup-length of $\ms1$ is at least $n$. Hence $\cat(\ms1)\geq n$. Now the theorem follows directly from \propref{p:closed}.
\end{proof}

\section{K-contact structures and Seifert $S^1$-bundles}\label{s:seifert}

To construct K-contact $(2n+1)$-manifolds with at least $(2n+1)$ closed characteristics, we use a particular type of {\it smooth} 4-orbifolds called {\it cyclic}. We refer to \cite{MRT} for a detailed description and terminology, as well as to  Appedix A.  Clearly, any smooth manifold yield a "trivial" orbifold structure, where each $\Gamma_{\ga}$ is trivial groups. On the other hand, generally, a smooth orbifold is not a smooth manifold, see \defref{d:orb}(b). However, there is a method to produce a smooth orbifold from a smooth manifold. This is given by~\cite[Proposition 4]{MRT} below.

\begin{prop} \label{prop:smooth->orb}
 Let $X$ be a smooth oriented $4$-manifold with embedded closed surfaces $D_i$ intersecting transversely. Given integers $m_i>1$ such that $\gcd(m_i,m_j)=1$ if $D_i\cap D_j \neq \emptyset $. Then there is a smooth orbifold $(X,\SU)$ with isotropy surfaces $D_i$ of multiplicities $m_i$.
\end{prop}

\begin{df}\rm Let $(X, \SU)$ be a cyclic, oriented $n$-dimensional orbifold. A {\em Seifert $S^1$-bundle} over $(X,\SU)$ is an oriented $(n+1)$-dimensional manifold $M$ equipped with a smooth $S^1$-action and a continuous map 
\[
\CD
\pi: M @>\text{quotient}>>M/S^1 \cong X
\endCD
\]
with the following property: for every orbifold chart $(\widetilde U,\phi,\mathbb{Z}_m)$, there is a commutative diagram
\[
\CD
(S^1\times\widetilde U)/\mathbb{Z}_m @>{\cong}>>\pi^{-1}(U)\\
@V{\pi}VV @V{\pi}VV\\
\widetilde U/\mathbb{Z}_m @>{\cong}>> U
\endCD
\]
where the action of $\mathbb{Z}_m$ on $S^1$ is by multiplication by $\zeta_m:=e^{2\pi i/m}$ and the top diffeomorphism is $S^1$-equivariant.
\end{df}

\begin{remark}\rm
We cohere the orientations of $X$ and $M$ as follows. Note that $S^1$ has a canonical (counterclockwise) orientation. Now, an orientation of $X$ gives an oriented orbifold chart $U$. We equip $S^1\times U$ with the product orientation. This gives us a local orientation on $M$, and so we get the desired orientation on $M$ (since $M$ is orientable because of the construction).  
\end{remark}

Suppose in the following that $(X,\SU)$ is a $4$-dimensional orbifold and $\pi: M\to X$ is a Seifert $S^1$-bundle over $X$.
According to the normal form of the $\ZZ_m$-action given in (\ref{action}), the open subset
$\pi^{-1}(U)\cong (S^1\times \widetilde U)/\ZZ_m$
is parametrized by $(u,z_1,z_2)\in S^1\x \CC^2$, modulo the $\ZZ_m$-action $\xi\cdot~(u,z_1,z_2) 
=(\xi u, \xi^{j_1} z_1, \xi^{j_2}z_2)$, for some integers $j_1,j_2$, where
$\xi=e^{2\pi i /m}$. The $S^1$-action on $M$ is given by
$s\cdot(u,z_1,z_2)=(su,z_1,z_2)$, so $\ZZ_m \subset S^1$ is the isotropy group  
and the exponents $j_1,j_2$ are determined by the $S^1$-action.

Following~\cite{MRT}, define an {\em orbit invariant} of the Seifert  $S^1$-bundle to be a finite set of triples $\{(D_i,m_i,j_i)\}$ where (see Definitions \ref{d:orient} and \ref{d:isotropy})
\begin{itemize}
\item $D_i$ is an isotropy surfase in $X$; 
\item $m_i\in \ZZ^+$ is the multiplicities $D_i$;
\item $j_i\in \ZZ^+$,
\end{itemize}
 and the local model around a point $p\in  
D_i^{\circ} =D_i - \bigcup_{i\neq j} (D_i \cap D_j)$ is of the form 
$(S^1\times \widetilde U)/\ZZ_{m_i}$ with action 
\[
\xi\cdot (u,z_1,z_2) =(\xi u, z_1, \xi^{j_i} z_2), \quad D_i=\{z_2=0\}.
\]
 If the orbifold is smooth, then 
for a point $p\in D_i\cap D_j$, the local model is of the form
$(S^1\times \widetilde U)/\ZZ_{m_i}$ with the action 
\[
\xi\cdot (u,z_1,z_2) =(\xi u, \xi^{j_j}z_1, \xi^{j_i} z_2),\quad  D_i=\{z_2=0\},\, D_j=\{z_1=0\}.
\] 

\begin{defin}\rm For a Seifert $S^1$-bundle $\pi: M\rightarrow X$, define its first Chern class as follows. Let $\mu=\mathbb{Z}_{m(X)}$, where $m(X)=\lcm\{m(x)\bigm|x\in X\}$. Consider the circle bundle $M/\mu\rightarrow X$ and its Chern class $c_1(M/\mu)\in H^2(X;\mathbb{Z})$. Define
\[
c_1(M/X)={1\over m(X)}c_1(M/\mu)\in H^2(X;\mathbb{Q}).
\]
\end{defin}

The next proposition is crucial in our constructions, since it shows that the orbit invariants determine the Seifert $S^1$-bundle globally, when $X$ is a smooth orbifold.

\begin{prop}\label{prop:existence-seifert} Let $(X,\SU)$ be an oriented $4$-manifold, and $D_i\subset X$ be closed surfaces in $X$ that intersect transversely. Let $m_i>1$ be such that $\gcd(m_i,m_j)=1$ if $D_i\cap D_j\neq\emptyset$. Let $0<j_i<m_i$ with $\gcd(j_i,m_i)=1$ for every $i$. Let $0<b_i<m_i$ be such that $j_ib_i=1\!\mod m_i$. Finally, let $B$ be a complex line bundle over $X$. Then there exists a Seifert $S^1$-bundle $f: M\rightarrow X$ with the orbit invariant $\{(D_i,m_i,j_i)\}$ and the first Chern class
\[
c_1(M/X)=c_1(B)+\sum_i{b_i\over m_i}[D_i].
\]
\end{prop}

\begin{proof} See~\cite[Proposition 14]{MRT}.
\end{proof}

\begin{thm}\label{thm:k-contact-seifert} Let $(X,\SU)$ be a symplectic $($almost K\"ahler$)$ cyclic orbifold. Let $[\omega]\in H^2(X;\mathbb{Q})$ be a symplectic form on $X$ and $\pi: M\rightarrow X$ be the Seifert $S^1$-bundle with $c_1(M/X)=[\omega]$. Then there exists a K-contact structure $(M,\ga, \xi,\Phi, g)$ such that $\pi^*(\omega)=d\ga$. 
\end{thm}

\begin{proof} 
See \cite[Theorem 20]{MRT}. 
\end{proof}

\section{Category weight and the main theorem}\label{s:wgt}

\begin{df}[\cite{Rud1}]\label{d:wgt}\rm Let $Y$ be a Hausdorff paracompact space, let $R$ be a commutative ring, and let $u \in H^q(Y;R)$ be an arbitrary element. We define the {\it  category weight} of $u$ (denoted by $\wgt u$) by setting
$$
\wgt u=\sup\{k\bigm|\varphi^*u=0 \text{ for every map $\varphi: A \to Y$ with } \cat \varphi<k\}
$$
where $A$ runs over all Hausdorff paracompact spaces.
\end{df}

\begin{lemma}\label{l:nonzero}
If $u\neq  0\in H^*(Y;R)$ then $\cat Y\geq \wgt u$.
\end{lemma}

\begin{proof}
This follows directly from \defref{d:wgt} of $\wgt$.
\end{proof}

\begin{lemma}\label{l:prod}
For all $u, v\in H^*(Y;R)$ we have $\wgt(u\smile v)\geq \wgt u + \wgt v$.
\end{lemma}

\begin{proof}
See~\cite[Theorem 1.5(v)]{RT}.
\end{proof}

\m Given a space $Y$ and an element $u\in H^n(X;R)$, the notation $u|_{\pi_n(Y)}=0$ means that
$$
\langle u, h(a)\rangle =0 \text{ for every }a\in \pi_n(Y)
$$
where $h: \pi_n(Y) \to H_n(Y)$ is the Hurewicz homomorphism and $\langle-.-\rangle$ is the Kronecker pairing. 

\begin{lemma}\label{l:wgt2} 
Let $Y$ be a finite $CW$-space, and let $u\in H^*(Y; R)$ be a cohomology class such that 
$u|_{\pi_2(Y)}=0$. Then $\wgt u \geq 2$.
\end{lemma}

\begin{proof}
See \cite[Lemma 2.1(ii)]{RT}.
\end{proof}

\m Given a symplectic orbifold $(X,\omega)$, we say that $(X,\omega)$ is {\em symplectically aspherical} if $[\omega]|_{\pi_2(X)}=0$.

 \begin{thm}\label{t:main}  Let $(X,\omega), \dim X=2n$ be a symplectic cyclic orbifold with $[\omega]\in H^2(X;\mathbb{Q})$, and let $\pi: M\rightarrow X$ be a Seifert $S^1$-bundle with $c_1(M/X)=[\omega]$. Assume that $(X,\omega)$ is symplectically aspherical. Then the total space $M$ of the Seifert fibration admits a K-contact structure with at least $2n+1$ closed characteristics.
 \end{thm}
  
 \begin{proof} It follows from \theoref{thm:k-contact-seifert} that $M$ admits a K-contact structure. So, it sufficies to prove that $\xi$ possesses at least $2n+1$ closed characteristics, where $\xi$ is the Reeb vector field of the K-contact structure.  Since $X=M/S^1$, and because of \propref{p:closed}, it sufficies to prove that $\cat X\geq 2n$. (Note that in this case $\cat X=2n$, because $\cat Y +1\leq \dim Y$ for all connected $Y$).
 
 \m 
 We have $\wgt[\omega]\geq 2$ by \lemref{l:wgt2}. Hence, by \lemref{l:prod}, 
 \[
 \wgt[\omega]^n\geq n\wgt [\omega]\geq 2n.
 \]
 Since $[\omega]^n\neq 0$ because of the symplecticity of the form $\omega$, we conclude that $\cat X\geq \wgt [\omega^n]\geq 2n$ by \lemref{l:nonzero}.
  \end{proof}

 \section{Examples of K-contact manifolds with at least $2n+1$ closed characteristics}\label{s:2n+1}

We present examples of K-contact manifods with at least $2n+1$ closed characteristics as total spaces of Seifert fibrations over symplectic orbifolds.  First, a big source of K-contact manifolds is given by~\cite[Theorem 6.1.26]{BG}. Such manifolds can be described as the total spaces of so-called Boothby--Wang fibrations. These are circle bundles over symplectic manifolds $(X,\omega)$ with integral symplectic form $\omega$. The first Chern class of such bundle is equal to the integral lift of the cohomology class $[\omega]$.  If one takes a symplectically aspherical base, one obtains an example of a K-contact manifold with at least $2n+1$ closed characteristics. 

\m To get a K-contact manifold which do not appear as a Boothby--Wang fibration, we can take any symplectically aspherical $4$-manifold with a collection of codimension 2 symplectic submanifolds which intersect transversely and choose the orbit invariant, according to  \propref{prop:existence-seifert} and \theoref{thm:k-contact-seifert}.
  
  \begin{example} {\rm Consider the Kodaira--Thurston manifold. It is defined as a nilmanifold of the form $N/\Gamma$ where $N=N_3\times \mathbb{R}$, $N_3$ is a 3-dimensional Heisenberg group and $\Gamma=\Gamma_3\times\mathbb{Z}$, $\Gamma_3$ denotes a co-compact lattice in $N_3$. It is easy to see and well known that $N/\Gamma$ is a symplectic manifold, and that $N/\Gamma$ is a symplectic fiber bundle over $\mathbb{T}^2$ with a symplectic fiber $F=\mathbb{T}^2$. Take, for example, the orbit invariat $(D,m,j)$ where $D=F$, $m$ is a chosen integer $m$, and $j<m$ with $\text{gcd}(j,m)=1$. It produced a Seifert $S^1$-bundle over a smooth orbifold $N/\Gamma$ determined by the cohomology class $[\omega]$ of the symplectic form $\omega$. The existence of the structure of the smooth orbifold determined by the data $(D,m,j)$ follows from~\propref{prop:smooth->orb}. By~\propref{prop:existence-seifert} there is a Seifert bundle determined by $[\omega]$. By~\theoref{thm:k-contact-seifert} there is a quasi-regular non-regular K-contact structure on the total space $M$ of this Seifert $S^1$-bundle over $X=N/\Gamma$. By~\theoref{t:main} $M$  admits a  K-contact structure with at least $2\cdot 2+1=5$ closed characteristics. The latter follows, since $N/\Gamma$ is aspherical.}
  \end{example}
  
\begin{rem} \rm Many more examples can be obtained by taking solvmanifolds $G/\Gamma$, that is, homogeneous spaces of solvable Lie groups $G$ over lattices $\Gamma\subset G$ (see, for example \cite{TO}, Chapter 3).
 \end{rem}
  
 \begin{example}{\rm Let $(X,\omega)$ be a closed symplectic manifold. Assume that the class $[\omega]$ is integral, and let $h$ denote its integral lift. Then, by a famous theorem of Donaldson \cite{D}, for $N$ large enough, the Poincar\'e dual of $Nh$ in $H_{2n-2}$ can be realized by a closed symplectic submanifold $D^{2n-2}\subset X^{2n}$. If one takes $(X,\omega)$ to be 4-dimensional and symplectically aspherical, and a surface $D\subset X$ guaranteed by Donaldson's theorem, then one obtains the necessary Seifert bundle choosing the orbit invariant $(D,m,j<m)$ in the same manner as in the previous example. The total space $M$ of this Seifert bundle admits a K-contact structure with at least $5=2\cdot 2+1$ closed characteristics.}
  \end{example}

  \section{Remarks on Rukimbira's theorem}\label{s:ruk}
  
  In~\cite[Theorem 2]{Ruk1} Rukimbira formulated a theorem that any closed manifold which carries a Sasakian structure with finite number of closed characteristics must have $b_1=0$. In~\cite{BG} (Theorem 7.4.8) this theorem is formulated for K-contact manifolds (not necessarily Sasakian), however, with a small inaccuracy: Rukimbira's proof goes through for K-contact manifolds (not necessarily Sasakian), except that the final result must be $b_1\leq 1$.
  
 \begin{thm}[\cite{Ruk1}]\label{t:ruk} Let $(M,\ga,\xi,\Phi, g)$ be a closed K-contact manifold where $\ga$ is a circle invariant form. Suppose that $\xi$ has only finite number of closed characteristics. Then $b_1(M)\leq 1$. If, in addition, $M$ is Sasakian then $b_1(M)=0$, 
 \end{thm}
  
 \begin{proof} For the convenience of the reader we give here a sketch (following Rukimbira). As in \secref{s:prel}, consider the function $S=i_Z(\ga)$. By \propref{p:closed}, critical circles of $S$ are closed characteristics of $\xi$. 
 Consider the set $F_{\xi}$ consisting of periodic points of $\xi$. It is proved in \cite{Ruk2} that $F_{\xi}$ is a union of closed characteristics and each component of $F_{\xi}$ is a totally geodesic odd-dimensional submanifold of $M$. Take a connected component of $F_{\xi}$ and call it $N$.  Following Rukimbira \cite[page 354]{Ruk1}, compute the Hessian $\operatorname{Hess}_S$ of $S$ on vectors orthogonal to $N$. Note that the calculation does not use the integrability condition, and is valid not only for Sasakian manifolds, but for arbitrary K-contact manifolds as well. It turns out  that for each critical point, $x\in N$, the Hessian, $d^2S_x$, of $S$ is non-degenerate
in directions normal to $N$ at $x$,  ($S$ is a clean function in terminology \cite{GS}) and each of its critical points has even index~\cite[Prop. 2]{Ruk1}). Again, the proof is based only on computing the Hessian, and is valid for any K-contact manifold.

\m By assumption, $\xi$ has only finite number of closed characteristics, and so $M$ has a finite number of critical $Z$-circles. Now, writing down the Morse inequalities as in~\cite[\S 5]{Ruk1}, one obtains
  $$
  \dim H^1(M,\mathbb{R})\leq 1.
  $$ 
  
\m
Finally, the first Betti number of any Sasakian manifold is even (the  zero case included), \cite{T}, and so $b_1(M)=0$ provided that $M$ is Sasakian.
 \end{proof}
 
\begin{cor} Let $(M,\ga,\xi,\Phi,g)$ be a closed K-contact manifold  where $\ga$ is a circle invariant form. If $b_1(M)>1$, then the Reeb vector field of the K-contact structure has infinite number of closed characteristics.
\end{cor}
 
\begin{question} {\rm If a closed K-contact $(2n+1)$-manifold admits exactly $n+1$ closed characteristics, then it is finitely covered by the sphere $S^{2n+1}$, Rukimbira \cite{Ruk3, Ruk4}. Can one characterize $(2n+1)$-manifolds which admit exactly $2n+1$ closed characteristics?}
\end{question}

\begin{question} \rm In view of \theoref{t:ruk}, we pose the following question: Are there examples of K-contact manifolds with $b_1(M)=1$ and with finite number of closed characteristics?
\end{question}

 \m If we remove the requirement of finiteness of closed characteristics, one can easily produce examples of K-contact closed manifolds with any prescribed $b_1(M)$ (and, moreover, any finitely presented $\pi_1(M)$), using the following two results.
 
 \begin{thm}[\cite{G}]\label{t:gompf} Any finitely presented group $\Gamma$ can be realized as a fundamental group of a closed symplectic $4$-dimensional manifold.
 \end{thm}
 
 \begin{thm}[\cite{HT}] Let $(X,\omega)$ be a closed symplectic manifold.  If 
 $$
 S^1\rightarrow P\rightarrow X
 $$
 is a Boothby--Wang fibration with the Chern class equal to $[\omega]$, then the total space $M$ of the  fiber bundle
 $$
 S^3\rightarrow P\times_{S^1}S^3\rightarrow X
 $$
 associated with $P\rightarrow X$ by the Hopf action of $S^1$ on $S^3$ admits a K-contact structure. 
 \end{thm}
 
 Now, we see that $M$ is a compact K-contact manifold with $\pi_1(M)=\pi_1(X)$. In particular, $b_1(M)=b_1(X)$, with any prescribed $b_1(X)$ because of \theoref{t:gompf}. 
 
\m  A more refined question comes from the fact that the full description of fundamental groups of symplectically aspherical manifolds is still an open question.  Therefore, it is not easy to describe Seifert bundles (and even principal circle bundles) over symplectically aspherical orbifolds (manifolds) with $b_1(M)\leq 1$. Partial results were obtained in~\cite{IKRT, KRT}. Below we give other examples. To do this, we need the Sullivan model theory. For the convenience of the reader, we expose the Sullivan theory in  Appendix B.
 
 \begin{example}\rm Here we recall some facts about symmetric spaces~\cite{He} and latttices in Lie groups~\cite{R}. Let $G$ be a simple non-compact Lie group and $K$ be a maximal compact subgroup of $G$. The homogeneous space $\hat X:=G/K$ is an irreducible  symmetric space. In particular, if one considers the case of {\it Hermitian symmetric space} $\widetilde X$, then it admits a Hermitian metric $h$, which is K\"ahler, $G$-invariant, and whose real part is a Riemannian metric $g$ on $\hat X$ as a Riemannian symmetric space. The list of such spaces is given in~\cite[Table II in Chapter IX]{He}. For example, the Hermitian symmetric spaces with classical $G$ are:
 $$SU(p,q)/S(U_p\times U_q), \,SO^*(2n)/U(n),$$
 $$ SO_0(p,q)/(SO(p)\times SO(q)),\quad Sp(n,\mathbb{R})/U(n).$$
 Hence, $h=g+i\omega$, where $\omega$ is a $G$-invariant K\"ahler (hence, symplectic)  form. There exists a co-compact lattice $\Gamma$ in $G$ which acts freely and properly on $G/K$. The latter is well known: on one hand, since $G$ is  simple, it admits a co-compact lattice ~\cite[Chapter 14]{R}, and on the other hand, since $K$ is a maximal compact subgroup, any such lattice admits a finite index sublattice acting freely on $G/K$ (Selberg's lemma). Therefore, $X=\Gamma\backslash{G\slash K}$ is a compact symplectic aspherical manifold. The total space $M$ of the Boothby-Wang fibration determined by $[\omega]$ is a K-contact manifold. Note that $\Gamma\subset G$ is a lattice in a simple Lie group $G$, therefore, $[\Gamma,\Gamma]=\Gamma$. Hence $b_1(X)=0$, since $\pi_1(X)\cong\Gamma$. It turns out that $b_1(M)=0$ as well. The latter can be seen, for example, as follows. Since $b_1(X)=0,$ the minimal model  $(\Lambda\,V,d)$ of $X$ does not have any generator of degree one, which is a cocycle. Since $M$ is a principal circle bundle over $X$, the relative Sullivan model of  this fiber bundle has  the form
 \begin{equation}
\label{relform}
 (\Lambda\,V,d)\rightarrow (\Lambda\,V\otimes\Lambda (t),D)\rightarrow (\Lambda (t),0),
 \end{equation}
 where $(\Lambda\,V,d)$ is the minimal model of $X$, and (over $\mathbb{Q}$)
 \[
 D|_{\Lambda\,V}=d,\,Dt=v,\,|v|=2, \,[v]=[\omega]\in H^2(X;\mathbb{Q})\cong H^2(\Lambda\,V,d).
\]
We have $[v]=[\omega]$ by~\cite[Example 4 in Section 15]{FHT}.
 Observe that the relative Sullivan model of the form \eqref{relform} cannot create elements of degree 1 which are cocycles. Hence, $b_1(M)=0$, as required.
\end{example} 

Note that the manifold $M$ is actually Sasakian, since $X$ is K\"ahler. 

\begin{question} Are there Seifert fibrations over symplectically aspherical manifolds which are K-contact, non-Sasakian, and admit K-contact structures with finite number of closed characteristics?
\end{question}  

\begin{appendices}
   
\section{Orbifolds}\label{sec:orbifolds}
\begin{df}[\cite{BG, MRT}]\label{d:orb}\rm (a) An $n$-dimensional (differentiable) orbifold is a pair $(X, \SU)$ where $X$ is a space and $\SU$ is an atlas $\SU=\{(\widetilde U_{\alpha},\phi_{\alpha}\Gamma_{\alpha})\}$. 
Here $\widetilde U_\alpha\subset\RR^n$, $U_\alpha\subset X$,  $\Gamma_{\alpha}$ is a finite subgroup of $GL(n,\RR)$, and 
$\phi_{\alpha}:\widetilde U_{\alpha} \to U_{\alpha}$ is a $\Gamma_{\alpha}$-invariant map which induces a 
homeomorphism $\widetilde U_{\alpha}/\Gamma_\alpha \cong U_{\alpha}$ onto an open set $U_\alpha$ of $X$. (Here $\Gamma_{\ga}$ acts linearly on $\widetilde U$ and trivially on $U$.)
There is also a condition of compatibility 
of charts: for each point $p \in U_{\alpha} \cap U_{\beta}$ there is some $U_\gamma\subset 
U_{\alpha} \cap U_{\beta}$ with $p \in U_\gamma$, monomorphisms
$\imath_{\gamma\alpha}: \Gamma_\gamma\inc \Gamma_\alpha$,
$\imath_{\gamma\beta}: \Gamma_\gamma\inc \Gamma_\beta$, and 
open embeddings $f_{\gamma\alpha}:\widetilde U_\gamma \to \widetilde U_\alpha$, 
$f_{\gamma\beta}:\widetilde U_\gamma \to \widetilde U_\beta$, which satisfy $\imath_{\gamma\alpha}(g)(f_{\gamma\alpha}(x))=
f_{\gamma\alpha}(g(x))$ and $\imath_{\gamma\beta}(g)(f_{\gamma\beta}(x))=
f_{\gamma\beta}(g(x))$, for $g\in \Gamma_\gamma$. 
 
\m For brevity, sometimes we will write $X$ instead of $(X,\SU)$ if there is no danger of confusion.

\m (b) We call $x\in X$ a {\em smooth point} if a neighborhood of $x$ is 
homeomorphic to a ball in $\RR^n$, and {\em singular} otherwise. 
We say that an orbifold $(X, \SU)$ is {\em smooth} if all its points are smooth. This is equivalent to say that $X$ is a topological manifold. 
\end{df}

\m As the groups $\Gamma_\alpha$ are finite, we can arrange (after conjugations if necessary) that $\Gamma_\alpha \subset \operatorname{O}(n)$. 

\begin{df}\label{d:orient}\rm The orbifold is {\em orientable} if $\Gamma_\alpha \subset SO(n)$ for all $\ga$ and the embeddings $f_{\gamma\alpha}$ preserve orientation. 
Note that for any point $x\in X$, we can always arrange  a chart $\phi:\widetilde U\to U$ with $\widetilde U\subset \RR^n$ being a 
ball centered at $0$ and $\phi(0)=x$, and $\widetilde U/\Gamma\cong U$, with $\Gamma \subset SO(n)$. In this case, we call $\Gamma$
the {\em isotropy group} at $x$. We call an orbifold {\em cyclic} if all its isotropy groups are cyclic groups $\Gamma\cong \ZZ_m$,
and $m=m(x)$ is the order of the isotropy at $x$. We call $x\in X$ a {\em regular point} if $m(x)=1$, otherwise we call it a (non-trivial) {\em isotropy point}. Clearly a regular point is smooth, but not the converse.
\end{df}

\m From here and till the end of the section, let $X$ be a 4-dimensional closed orientable cyclic orbifold. Take $x\in X$ and a chart $\phi:\widetilde U\to U$ around $x$. Let 
$\Gamma=\mathbb{Z}_m\subset SO(4)$ be the isotropy group. Then $U$ is homeomorphic to an open neighborhood of 
$0\in \mathbb{R}^4/\mathbb{Z}_m$. A matrix of finite order in $SO(4)$ is conjugate to a diagonal matrix in $U(2)$ of
the type 
\[
(\exp(2\pi i j_1/m),\exp(2\pi i j_2/m))=(\xi^{j_1},\xi^{j_2}),
\]
 where $\xi=e^{2\pi i/m}$. Therefore we can
suppose that $\widetilde U\subset \CC^2$ and $\Gamma=\ZZ_m=\la \xi\ra\subset U(2)$ acts on $\widetilde U$ as
 \begin{equation}\label{action}
   \xi \cdot (z_1,z_2) := (\xi^{j_1} z_1,\xi^{j_2}z_2).
 \end{equation}
Here $j_1,j_2$ are defined modulo $m$.
As the action is effective, we have $\gcd(j_1,j_2,m)=1$. 

\begin{df}\label{d:isotropy}\rm We say that $D\subset X$ is an {\em isotropy surface of multiplicity $m$} if 
 \begin{itemize}
 \item
 $D$ is closed;
 \item there is a dense open subset $D^\circ\subset D$ which is a surface;
 \item $m(x)=m$ for all $x\in D^\circ$.
\end{itemize}
\end{df}

\begin{prop}\label{prop:models}
Let $X$ be as above and $x\in X$ with local model $\CC^2/\ZZ_m$. Then there
are at most two isotropy surfaces $D_i$, with multiplicity $m_i | m$, through $x$. If there are two such
surfaces $D_1,D_2$, then they intersect transversely and $\gcd(m_1,m_1)=1$. The point is smooth if and only if $m_1m_2=m$.
\end{prop}

\begin{proof}
See~\cite[Proposition 2]{MRT}.
\end{proof}

Now we want to describe the local models for the $\mathbb{Z}_m$-action in two cases used in this work. For all other possibilities we refer to \cite{MRT}. For an action given by \eqref{action}, we 
set $m_1:=\gcd(j_1,m)$, $m_2:=\gcd(j_2,m)$. Note that $\gcd(m_1,m_2)=1$, so we can
write $m_1m_2 d=m$, for some integer $d$. Put $j_1=m_1 e_1$, $j_2=m_2 e_2$, $m=m_1 c_1=m_2 c_2$.
Clearly $c_1=m_2d$ and $c_2=m_1d$ with $d=\gcd(c_1,c_2)$.

In what follows we  consider the following  cases:

\begin{enumerate}
\item[(a)] There are two isotropy surfaces and $x$ is a smooth point, $m_1,m_2>1$, $d=1$. 
Let us see that the action is equivalent to the product of one action on each factor $\mathbb{C}$. 
In this case $c_2=m_1$ and $c_1=m_2$. So $\gcd (c_1,c_2)=1$ and $m=c_1c_2$. The action is given by
\[
\xi \cdot (z_1,z_2) := (\exp(2\pi i e_1/c_1)z_1,\exp(2\pi i e_2/c_2)z_2).
\]
 We see that 
 $$
 \begin{array}{ccc}
 \xi^{c_1}\cdot (z_1,z_2)=(z_1, \exp(2 \pi i c_1 e_2/c_2) z_2),  \\
 \xi^{c_2}\cdot (z_1,z_2)=(\exp(2 \pi i c_2 e_1/c_1) z_1, z_2),
 \end{array} 
 $$
so, the surfaces $D_1=\{(z_1,0)\}$ and $D_2=\{(0,z_2)\}$ have isotropy groups $\la \xi^{c_1}\ra=\mathbb{Z}_{m_1}$ 
and $\la\xi^{c_2}\ra=\mathbb{Z}_{m_2}$, respectively. In this case $m=m_1m_2$, $d=1$.

Note that $\ZZ_m=\la \xi^{c_1}\ra \times \la\xi^{c_2}\ra$ if and only if $d=\gcd(c_1,c_2)=1$. 
In this case the action of $\mathbb{Z}_m$ decomposes as the product of the actions of 
$\mathbb{Z}_{m_2}$ and $\mathbb{Z}_{m_1}$ on each of the factors $\CC$. The quotient space is
$\CC^2/\ZZ_m\cong \CC/\ZZ_{m_2} \x \CC/\ZZ_{m_1}$, which is homeomorphic to $\CC\x\CC$, and hence $x$ is a
smooth point.

\item[(b)]  There is exactly one isotropy surface and $x$ is a smooth point. In this case $m_2=1$ and $m_1=m$. As $d=1$, this is basically
as case (b). The action is $\xi \cdot (z_1,z_2)=(z_1,\exp(2\pi j_2/m)  z_2)$. There is only one surface 
$D_1=\{(z_1,0)\}$ with non-trivial isotropy $m$, and all its points have the same 
isotropy. The quotient $\CC^2/\ZZ_m =\CC \x (\CC/\ZZ_m)$ is topologically smooth.
\end{enumerate}

In case (a), we can change the generator $\xi=e^{2\pi i/m}$ of $\ZZ_{m}$ to 
$\xi'=\xi^k$ for $k$ such that $k e_i \equiv 1 \pmod{m_i}$, $i=1,2$,
so that 
\[
\xi' \cdot (z_1,z_2)=\left(\exp\left(\frac{2 \pi i}{m_2}\right)z_1, \exp\left(\frac{2 \pi i}{m_1}\right)z_2\right).
\] 
With this new generator, the action has model $\CC^2$ with the action 
\[
\xi\cdot (z_1,z_2)=(\xi^{m_1}z_1, \xi^{m_2} z_2),\quad \xi=e^{2\pi i/m}.
\]
 Similar remark applies to case (b).

\section{Sullivan models}\label{sec:sullivan}

We use Sullivan models of fibrations as a tool of calculating
cohomology in some examples. In the sequel our notation follows
\cite{FHT}. In this section we denote by $\mathbb{K}$ an arbitrary field of
characteristic zero. We consider the
category of commutative graded differential algebras (or, in the
terminology of~\cite{FHT}, cochain algebras). If $(A,d)$ is a cochain
algebra with a grading $A=\oplus_pA^p$, the degree $p$ of $a\in A^p$
is denoted by $|a|$. 

\m Given a graded vector space $V,$ consider the algebra $\Lambda
V=S(V^{\operatorname{even}})\otimes\Lambda (V^{\operatorname{odd}})$, that is, $\Lambda V$ denotes
a free algebra which is a tensor product of a symmetric algebra over
the vector space $V^{\operatorname{even}}$ of elements of even degrees, and an
exterior algebra over the vector space $V^{\operatorname{odd}}$ of elements of odd
degrees.

\m We will use the following notation:
\begin{itemize}
\item by $\Lambda V^{\leq p}$ and $\Lambda V^{>p}$ are denoted the
subalgebras generated by elements of degree $\leq p$ and of degree
$>p,$ respectively;
\item if $v\in V$ is a generator, $\Lambda v$ denotes
the subalgebra generated by $v\in V$,

\item $\Lambda^pV=\langle v_1,\ldots v_p\rangle$,
 $\Lambda^{\geq q}V=\oplus_{i\geq
q}\Lambda^iV$, $\Lambda^+V=\Lambda^{\geq 1}V.$
\end{itemize}

\begin{df}\rm A {\it Sullivan algebra} is a
commutative graded differential algebra of the form $(\Lambda V,d)$,
where
\begin{itemize}
\item $V=\oplus_{p\geq 1}V^p$;
\item $V$ admits an increasing  filtration
\[
V(0)\subset V(1)\subset \cdots\subset V=\bigcup_{k=0}^{\infty}V(k)
\]
such that $d=0$ on $V(0)$ and 
$d: V(k)\rightarrow \Lambda V(k-1),\,k\geq 1$.
\end{itemize}
\end{df}

\begin{df}\rm  A Sullivan algebra $(\Lambda
V,d)$ is called {\it minimal}, if
\[
\operatorname{Im}\, d\subset \Lambda^+ V\cdot\Lambda^+V.
\]
\end{df}

\begin{df}\rm  A {\it Sullivan model} of
a commutative graded differential algebra $(A,d_A)$ is a morphism
$$m: (\Lambda V,d)\rightarrow (A,d_A)$$
inducing an isomorphism $m^*: H^*(\Lambda V,d)\rightarrow
H^*(A,d_A)$. 
\end{df}

If $X$ is a CW-complex, there is a cochain algebra
$(A_{PL}(X),d_A)$ of polynomial differential forms. For a smooth
manifold $X$ we take a smooth triangulation of $X$ and as  the model
of $X$ the Sullivan model of $A_{PL}(X).$ If it is minimal, it is
called the {\it Sullivan minimal model of} $X$. 
It is well known (see \cite{FHT}, Proposition 12.1) that any commutative cochain algebra $(A,d)$ satisfying $H^0(A,d)=\mathbb{K}$ admits a Sullivan model.

\begin{defin}\rm A {\it relative Sullivan algebra}
is a graded commutative differential algebra of the form $(B\otimes
\Lambda V,d)$ such that
\begin{itemize}
\item $(B,d)=(B\otimes 1,d), H^0(B)=\mathbb{K},$
\item $1\otimes V=V=\oplus_{p\geq 1}V^p$,
\item $V=\cup_{k=0}^{\infty} V(k), \text{where } V(0)\subset V(1)\subset\cdots $,
\item $d: V(0)\rightarrow B,\, d:V(k)\rightarrow
B\otimes\Lambda V(k-1),\quad k\geq 1$
\end{itemize}
\end{defin}
Relative Sullivan algebras are models of fibrations. Let
$p: X\rightarrow Y$ be a Serre fibration with the homotopy fiber
$F$. Choose Sullivan models
\[
m_Y: (\Lambda V_Y,d)\rightarrow (A_{PL}(Y), d_Y),\quad \overline m: (\Lambda V,\ov d)\rightarrow A_{PL}(F).
\]

There is a commutative diagram of cochain algebra morphisms
\[
\CD
A_{PL}(Y) @>>> A_{PL}(X) @>>> A_{PL}(F)\\
@A{m_Y}AA @A{m}AA @A{\overline m}AA\\
(\Lambda V_Y,d) @>>> (\Lambda V_Y\otimes\Lambda V,d) @>>> (\Lambda
V,\ov d)
\endCD
\]
in which $m_Y,m,\overline m$ are all Sullivan models (see~\cite[Propositions
15.5 and 15.6]{FHT}).

   \end{appendices}

\end{document}